\documentclass[11pt]{amsart}
\usepackage{amsthm}
\usepackage{amssymb}
\usepackage{amsmath}
\usepackage{amsfonts}
\usepackage{graphicx} 
\usepackage{mathtools}
\usepackage{soul}
\usepackage{ytableau,tikz,varwidth}
\usepackage{latexsym,amssymb,amsmath,hyperref,amsthm,amsfonts,
caption,tikz,comment}
\usetikzlibrary{calc}
\usepackage{tikz-cd}
\usepackage{enumitem}
\usetikzlibrary{graphs}
\usetikzlibrary{graphs.standard}
\usetikzlibrary{shapes.misc}
\usetikzlibrary {shapes.geometric} 
\usepackage{amsrefs}
\usepackage{latexsym,hyperref,caption,subcaption,comment}
\usepackage{nicematrix}
\usepackage{blkarray, bigstrut}

\numberwithin{equation}{section}
\tikzset{
  curarrow/.style={
  rounded corners=8pt,
  execute at begin to={every node/.style={fill=red}},
    to path={-- ([xshift=-50pt]\tikztostart.center)
    |- (#1) node[fill=white] {$\scriptstyle \delta$}
    -| ([xshift=50pt]\tikztotarget.center)
    -- (\tikztotarget)}\
    }
}

\tikzcdset{diagrams={arrows={shorten >=-0.5ex,shorten <=-0.5ex}}}

\newtheorem{theorem}{Theorem}[section]

\newtheorem{proposition}[theorem]{Proposition}

\theoremstyle{definition}

\newtheorem{remark}[theorem]{Remark}

\def\opn#1#2{\def#1{\operatorname{#2}}} 
\opn\Cl{Cl} \opn\pdim{pdim} \opn\Im{Im} \opn\Ker{Ker} \opn\ini{in} \opn\typ{type}

\DeclareOldFontCommand{\rm}{\normalfont\rmfamily}{\mathrm}

\begin{document}

\title {Combinatorial Identities Using the Matrix Tree Theorem}
\author{Nayana Shibu Deepthi}
\address{Department of Mathematical Sciences, IISER Mohali, Knowledge City, Sector 81, SAS Nagar, Punjab, 140-306, India}
\email{nayanashibu@iisermohali.ac.in}
\author{Chanchal Kumar}
\address{Department of Mathematical Sciences, IISER Mohali, Knowledge City, Sector 81, SAS Nagar, Punjab, 140-306, India}
\email{chanchal@iisermohali.ac.in}


\begin{abstract}
In this paper, we explore some interesting applications of the matrix tree theorem.
In particular, we present a combinatorial interpretation of a distribution of $(n-1)^{n-1}$, in the context of uprooted spanning trees of the complete graph $K_{n}$, which was previously obtained by Chauve--Dulucq--Guibert.
Additionally, we establish a combinatorial explanation for the distribution of $m^{n-1}n^{m-1}$, related to spanning trees of the complete bipartite graph $K_{m,n}$, which seems new.
Furthermore, we extend this study to the graph $K_{n}\setminus \{e_{1,n}\}$, obtained by deleting an edge from $K_n$, and derive a new identity for the number of its uprooted spanning trees.
\end{abstract}

\maketitle


\section{Introduction and preliminaries}\label{sec:intro}

Throughout this work, we deal with finite simple labeled graphs, where each vertex has a unique label and there are no multiple edges or loops.
Counting spanning trees is a significant problem in combinatorics, with numerous applications in networks and other areas.
The matrix tree theorem, initially formulated by Kirchhoff \cite{Kirchhoff}, states that the number of spanning trees of a graph can be determined using the determinant of a modified version of its Laplacian matrix -- a matrix encoding the structure of the graph.
A recent elementary proof of the matrix tree theorem can be found in \cite{AlgComb}.
Further, much work has been performed toward understanding the extension of the matrix tree theorem to directed graphs, weighted graphs, and so on. (For examples, see \cites{Tutte,Chaiken,Patrick}.)
In this paper, we explore interesting applications of this theorem in enumerating uprooted spanning trees in complete graphs $K_{n}$ and complete bipartite graphs $K_{m,n}$.

Trees were first extensively studied by Cayley \cite{Cayley}, who established the well-known formula $n^{n-2}$ for the number of labeled trees with $n$ vertices. 
Since every spanning tree of $K_{n}$ is a labeled tree with $n$ vertices, the same formula gives the number of spanning trees of the complete graph $K_{n}$.
The number of uprooted spanning trees of a complete graph $K_{n}$ with $n$ vertices is $(n-1)^{n-1}$.
Chauve--Dulucq--Guibert \cite{CDG} provided combinatorial interpretations of various identities involving $(n-1)^{n-1}$.
The number of spanning trees of a complete bipartite graph $K_{m,n}$ is given by $m^{n-1}n^{m-1}$ (See, \cites{Austin, Abu, JinLiu}).
For modern expositions, these formulas can also be found in textbooks such as \cites{enumerative, Harary}.

Through our exploration, we derived two key identities related to the enumeration of spanning trees.
The first identity pertains to the complete graph $K_{n}$, and although it has been previously obtained in \cite{CDG}, we present an alternative approach to derive the same result. 
The second identity, which arises in the context of complete bipartite graphs $K_{m,n}$, appears to be new. 
Specifically, we provide a combinatorial interpretation of an identity involving $m^{n-1}n^{m-1}$, which represents the number of spanning trees in $K_{m,n}$.


Now, we provide the basic definitions and notation necessary to understand our results.
For any undefined terms and notations in this article, see \cites{Aigner, Bapat}.

Throughout this paper, unless stated otherwise, we consider the finite simple graph $G = (V(G), E(G))$, with the vertex set $V(G)=[n]=\{1,2,\dots n\}$.
For any $v\in V(G)$, the number of edges adjacent to vertex $v$ is called the {\it degree} of $v$, denoted by $\deg_{G}(v)$. 
A graph $G=(V(G),E(G))$ is completely determined by its {\it adjacency matrix} $A(G) = [a_{ij}]_{1\leq i,j\leq n}$, where $a_{ij}$ denotes the number of edges from $i$ to $j$.
We assume $G$ to be loopless, therefore $a_{ii}=0$ for all $i\in V(G)$. 
Note that $G$ is simple if $a_{ij}=0$ or $1$ for all $i,j\in V(G)$.
For a vertex $v\in V(G)$, we denote by $G\backslash v$ the graph obtained from $G$ by deleting the vertex $v$ along with all edges incident to it. 
Similarly, for a subset $E^{\prime}\subseteq E(G)$, we denote by $G\backslash E^{\prime}$ the graph obtained from $G$ by deleting the edges $e\in E^{\prime}$ while keeping all vertices intact.

For a graph $G$ with $V(G)=[n]$, the {\it degree matrix} of $G$ is defined as the diagonal matrix $D(G)={\rm diag}[\deg_{G}(1) ,\deg_{G}(2), \dots , \deg_{G}(n)]$.
The matrix $L(G) = D(G) - A(G)$ is called the {\it Laplacian matrix} of $G$.
Upon deleting any $i$-th row and column of the Laplacian matrix $L(G)$, we obtain the {\it reduced} Laplacian matrix of $G$, denoted by $\widetilde{L}(G)$.
Note that $L(G)$ is singular, and every $(i,j)$-th cofactor of $L(G)$ is equal.
Let ${\rm SPT}(G)$ denote the set of all spanning trees of the graph $G$.
Then, the cardinality of ${\rm SPT}(G)$ is given by the following theorem.

\begin{theorem}[The matrix tree theorem {\cite{Kirchhoff}}]\label{Thm:MTT}
Let $G$ be a simple connected graph with the vertex set $[n]$. Then the number of spanning trees of $G$ is given by 
\[
|{\rm SPT}(G)| = \det(\widetilde{L}(G)).
\]
\end{theorem}

\noindent Moreover, the matrix tree theorem remains valid for multigraphs without loops. 
In this context, multiple edges between the same pairs of vertices are allowed. 
However, loops are excluded from consideration, as they do not contribute to the enumeration of spanning trees.
For more details, see \cite{Aigner}. 

As a direct application of Theorem~\ref{Thm:MTT}, one can compute the number of spanning trees of complete graphs and complete bipartite graphs as follows:
\[
|{\rm SPT}(K_n)| = \det(\widetilde{L}(K_n)) = n^{n-2}, \quad \text{and} \quad |{\rm SPT}(K_{m,n})| = \det(\widetilde{L}(K_{m,n})) = m^{n-1} n^{m-1}.
\]

Suppose that in a graph $G$, an edge $e \in E(G)$ is chosen, and we denote the resulting marked graph as $(G; e)$.
Let ${\rm SPT}(G; e)$ denote the set of all spanning trees of $G$ that contain the edge $e$.
Let $G^{\prime}=G \setminus \{e\}$ be the graph obtained by deleting the edge $e$ from $G$.
Then we have the decomposition
\begin{equation}\label{eq:xSPT}
    |{\rm SPT}(G)| =  |{\rm SPT}(G^{\prime})| + |{\rm SPT}(G;e)|. 
\end{equation}
Let $e=\{i_{0},j_{0}\}\in E(G)$ and in the marked graph $(G;e)$, the edge $e$ be assigned a variable weight $x$.
Let $L_{x}(G;e)$ denote the Laplacian matrix of $(G;e)$, and $\widetilde{L}_{x}(G;e)$ denote its reduced Laplacian matrix.
Then the entries of the Laplacian matrix $L_{x}(G;e)$ are as follows. 
For $i\in\{i_{0},j_{0}\}$, the diagonal entry corresponding to the vertex $i$ is equal to $\deg_{G^{\prime}}(i)+x$.
The non-diagonal entries at positions $(i_{0},j_{0})$ and $(j_{0},i_{0})$ are equal to $-x$, while all other entries are the same as those in the Laplacian matrix of $G$.

We shall now modify the matrix tree theorem to count the number of spanning trees of a simple graph that contain a specified edge.

\begin{theorem}[Modified matrix tree theorem]\label{thm:modifiedMTT}
Let $(G; e)$ be the graph $G$ with a specific edge $e\in E(G)$, to which we assign a variable weight $x$. Then, the number of spanning trees of $G$ that contain the edge $e$ is equal to the coefficient of $x$ in $\det(\widetilde{L}_{x}(G;e))$.   
\end{theorem}

\begin{proof}
The determinant $\det(\widetilde{L}_{x}(G;e))$ is a linear polynomial in $x$.
That is, we can write $\det(\widetilde{L}_{x}(G;e))=\alpha+\beta x$, where $\alpha$ and $\beta$ are constants independent of $x$.
Note that when $x=0$, the reduced Laplacian matrix $\widetilde{L}_{0}(G;e)$ is equal to the reduced Laplacian matrix of the graph $G^{\prime}=G\setminus\{e\}$, that is, $\widetilde{L}_{0}(G;e)=\widetilde{L}(G^{\prime})$.
Similarly, when $x=1$, $\widetilde{L}_{1}(G;e)$ is equal to the reduced Laplacian matrix of graph $G$, i.e., the matrix $\widetilde{L}_{1}(G;e)=\widetilde{L}(G)$.
Therefore, we have
\[
\det(\widetilde{L}_{0}(G;e)) = \alpha = |{\rm SPT}(G^{\prime})|, \quad \text{and} \quad \det(\widetilde{L}_{1}(G;e)) = \alpha + \beta = |{\rm SPT}(G)|.
\]
Thus, using \eqref{eq:xSPT} and the above observations, we obtain
\begin{align*}
|{\rm SPT}(G;e)|  = |{\rm SPT}(G)| - |{\rm SPT}(G^{\prime})|= \beta. 
\end{align*}
Hence, the number of spanning trees of $G$ that contain the edge $e$ is equal to the coefficient of  $x$ in $\det(\widetilde{L}_{x}(G;e))$.
\end{proof}

A {\it rooted} tree with the vertex set $[n]$ is a tree in which a specific vertex, called the {\it root}, is distinguished, defining a natural orientation for the graph.
We assume that all trees $T \in \rm{SPT}(G)$ are oriented such that every edge is directed away from the root $r$.
Consequently, for each vertex $v\in V(G)$ with $v\neq r$, there exists a unique directed path from $r$ to $v$.
Any vertex $u\neq r$ that appears in this unique path from $r$ to $v$ is referred to as an {\it ancestor} of $v$, while $v$ is called a {\it descendant} of $u$.
If $v$ is a descendant of $u$ and no other vertex in $V(G)$ lies between them on this unique path in $T$, then $v$ is said to be a {\it child} of $u$, and equivalently, $u$ is the {\it parent} of $v$.
In $T$, each vertex $v\neq r$ has exactly one parent.
A rooted tree $T$ with the vertex set $[n]$ is called an {\it uprooted tree} if the root is greater than all of its children.
For a graph $G$ with $V(G)=[n]$, the collection of all uprooted spanning trees of $G$ is denoted by $\mathcal{U}_{G}$.


\section{\texorpdfstring{Distribution of $(n-1)^{n-1}$ using the matrix tree theorem}{Distribution of Kn using the matrix tree theorem }}\label{sec:Kn}

Chauve, Dulucq, and Guibert \cite{CDG} proved the following identity and provided a combinatorial interpretation.
For $n\geq 2$,
\begin{equation}\label{eq:identity1}
(n-1)^{n-1}=\sum_{k=0}^{n-1}(n-1-k)~n^{n-2-k}~(n-1)^{k-1}.
\end{equation}

Let $\mathcal{U}_{K_n}$ denote the set of all uprooted spanning trees of the complete graph $K_{n}$ and let $\mathcal{A}_n$ denote the set of rooted spanning trees of $K_n$ in which vertex $1$ be a non-root leaf.
Chauve et al. \cite{CDG} constructed an elegant bijection between certain subsets of labeled rooted spanning trees of $K_n$, resulting in a bijection $\varphi \colon \mathcal{U}_{K_n} \longrightarrow\mathcal{A}_{n}$.
A brief description of the bijection $\varphi$ is also given in \cite{CGS}.
Each rooted tree in $\mathcal{A}_n$ can be constructed by first choosing a labeled spanning tree with the vertex set $\{2,3,\dots,n\}$ in $(n-1)^{n-3}$ ways, then selecting a root from these $(n-1)$ vertices, and finally attaching the leaf $1$ to any of the remaining $(n-1)$ vertices. Hence, $|\mathcal{A}_n|=(n-1)^{n-3}(n-1)(n-1)=(n-1)^{n-1}$.
Furthermore, the Pr\"ufer sequence of any tree in $\mathcal{A}_n$ is a sequence of length $n-1$ with entries from $\{2,3,\dots , n\}$, where the last term of the sequence corresponds to the root.
This characterization provides a direct proof that $|\mathcal{U}_{K_n}| = (n-1)^{n-1}$.

Let $\mathcal{T}_{r}\subseteq\mathcal{U}_{K_{n}}$ denote the set of all uprooted spanning trees of $K_n$ with root $r$.
For $n\geq 2$ and $0\leq k\leq n-1$, it is shown in \cite{CDG}, using a modified Pr\"ufer sequence, that $$|\mathcal{T}_{n-k}|=(n-1-k)~n^{n-2-k}~(n-1)^{k-1}.$$
As $\mathcal{U}_{K_{n}}= \coprod_{k=0}^{n-1} \mathcal{T}_{n-k}$, this provides a combinatorial explanation of the identity~\eqref{eq:identity1}. 

In this section, we present an alternative derivation of the formula for $|\mathcal{T}_{n-k}|$ as an application of the matrix tree theorem.

\begin{theorem}[Chauve--Dulucq--Guibert]\label{thm:Kn}
For $n\geq 2$ and $0\leq k\leq n-1$, let $\mathcal{T}_{n-k}\subseteq \mathcal{U}_{K_{n}}$ be the set of all the uprooted spanning trees of $K_n$ with root $n-k$. Then we have
\[
|\mathcal{T}_{n-k}|=(n-1-k)~n^{n-2-k}~(n-1)^{k-1}. 
\]
\end{theorem}

\begin{proof}
Let us consider $\mathcal{T}_{n}$, the set of all uprooted spanning trees of $K_n$ with root $n$. 
For any tree $T\in {\rm SPT}(K_n)$, designating $n$ as the root results in an uprooted spanning tree of $K_n$ with root $n$. 
Thus, the number of such trees is given by $|\mathcal{T}_{n}|=|{\rm SPT}(K_n)|=n^{n-2}$.

Now, we extend our analysis to the general case by considering $\mathcal{T}_{n-k}$, the set of all uprooted spanning trees of $K_n$ with root $n-k$, where $0\leq k\leq n-1$.
For any tree $T\in\mathcal{T}_{n-k}$, observe that the root $n-k$ is not adjacent to any vertex $i$ where $i > n-k$. 
To account for this, we define the graph $$G_{n-k}= K_n\backslash \{e_{n-k,i}\colon n-k+1\leq i\leq n\},$$ where $e_{u,v}$ represents the edge $\{u,v\}\in E(K_n)$.
By construction, $G_{n-k}$ ensures that the vertex $n-k$ has no edges connecting it to any vertex $i$ for $ n-k+1\leq i\leq n$.
Thus, choosing any spanning tree of $G_{n-k}$ and designating $ n-k$ as the root results in an uprooted spanning tree of $K_{n}$ with root $n-k$. 
Therefore, we obtain 
\begin{equation}\label{eq:01}
      |\mathcal{T}_{n-k}|=|{\rm SPT}(G_{n-k})|, \ \text{ for all } 0\leq k\leq n-1.
\end{equation}

To determine the number of spanning trees of $G_{n-k}$, we apply the matrix tree theorem.
We obtain the reduced Laplacian matrix of $G_{n-k}$ by deleting the row and column corresponding to the root $n-k$ from the Laplacian matrix $L(G_{n-k})$.
The resulting reduced Laplacian matrix $\widetilde{L}(G_{n-k})$ of $G_{n-k}$ is an $(n-1)\times (n-1)$ matrix, where the $i$-th diagonal entry is equal to $n-1$ for $1 \leq i \leq n-k-1$, and equal to $n-2$ for $n-k \leq i \leq n-1$, while all non-diagonal entries are $-1$.

Let $R_{i}$ denote the $i$-th row and $C_{j}$ denote the $j$-th column of the matrix $\widetilde{L}(G_{n-k})$.
Let us perform the following series of row and column operations on $\widetilde{L}(G_{n-k})$.
\begin{enumerate}
    \item $C_1 \longrightarrow C_1 + \sum_{i=2}^{n-1} C_{i}$,
    \item $R_i \longrightarrow R_i - R_{n-k}$, for all $ 1\leq i\leq  n-k-1$,
    \item $R_i \longrightarrow R_i - R_1$, for all $2\leq i\leq  n-k-1$, and then
    \item For each $n-k \leq i\leq n-1$, perform $R_i \longrightarrow R_i + \frac{1}{n}R_{j}$, for all $2\leq j\leq  n-k-1$.
\end{enumerate}
After applying the above row and column operations in the given order, the matrix $\widetilde{L}(G_{n-k})$ reduces to the following block matrix
\[
\widetilde{L}(G_{n-k})^{\prime} = 
\begin{bNiceArray}{c|c}[margin,columns-width=auto]
  A & B \\
  \hline
 C & D
\end{bNiceArray}_{(n-1)\times(n-1)},
\]
where $A={\rm diag}[1 ,n, \dots , n]$ is an $(n-k-1)\times (n-k-1)$ diagonal matrix, $B$ is an $(n-k-1)\times k$ matrix with all entries zero except for the first entry, which is $-(n-1)$, $C$ is a $k\times (n-k-1)$ zero matrix and $D$ is a $k\times k$ matrix with diagonal entries equal to $n-2$ and non-diagonal entries equal to $-1$.

As the matrix $\widetilde{L}(G_{n-k})^{\prime}$ is an upper triangular block matrix, its determinant is given by ${\rm det}(\widetilde{L}(G_{n-k})^{\prime})= {\rm det}(A)\cdot{\rm det}(D)$.  
Since $A$ is a diagonal matrix, its determinant is the product of its diagonal entries, yielding ${\rm det}(A)=n^{n-k-2}$.  
Now, let us compute the determinant of the matrix $D$.
Performing the column operation $C_1 \longrightarrow C_1 + \sum_{i=2}^{k} C_{i}$ followed by the row operations $R_i \longrightarrow R_i - R_1$, for all $2 \leq i\leq k-1$, on the matrix $D$, we transform it into the matrix $D^{\prime}$ given by
\[
D^{\prime} =
\begin{blockarray}{ccccccc}
\begin{block}{[cccccc]c}
n-k-1& -1 & -1& \cdots  & -1 & -1  \\
0 & n-1 &  0  & \cdots & 0& 0 \\
0  & 0 & n-1 & \cdots & 0 & 0\\
\vdots & \vdots &\vdots & \ddots & \vdots & \vdots\\
0& 0 &0 &\cdots &n-1 & 0\\
0& 0 & 0 &\cdots & 0 &n-1  & \text{\scriptsize $k \times k$} & \\
\end{block}
\end{blockarray}. 
\]
The determinant of $D^{\prime}$ is clearly ${\rm det}(D^{\prime})=(n-k-1)(n-1)^{k-1}$. 
Since row and column operations do not change the determinant, we have $${\rm det}(D)={\rm det}(D^{\prime})= (n-k-1)(n-1)^{k-1}.$$
Therefore, we have ${\rm det}(\widetilde{L}(G_{n-k})^{\prime})=n^{n-k-2}~(n-k-1)~(n-1)^{k-1}.$
By the matrix tree theorem, the number of spanning trees of $G_{n-k}$ is  
\begin{equation}\label{eq:02}
     |{\rm SPT}(G_{n-k})| = {\rm det}(\widetilde{L}(G_{n-k})^{\prime})=n^{n-k-2}~(n-k-1)~(n-1)^{k-1}, \ \text{ for all } 0\leq k\leq n-1.
\end{equation}
Thus, from \eqref{eq:01} and \eqref{eq:02}, we have 
\begin{align*}
|\mathcal{T}_{n-k}|=(n-1-k)~n^{n-2-k}~(n-1)^{k-1}, \ \text{ for all } 0\leq k\leq n-1. 
\end{align*}
This concludes our proof.
\end{proof}

A refinement of the enumeration of uprooted spanning trees obtained in Theorem~\ref{thm:Kn} will be presented in Section~\ref{sec:refinement}.


\section{\texorpdfstring{Distribution of $m^{n-1}n^{m-1}$ using the matrix tree theorem}{Distribution of Kmn using the matrix tree theorem }}\label{sec:Kmn}

In this section, we explore the enumeration of spanning trees of the complete bipartite graph $K_{m,n}$. 
In this process, we count the number of uprooted spanning trees of $K_{m,n}$ with root $m+n$ using the matrix tree theorem and derive the following new identity
\begin{equation}\label{eq:identity2} 
m^{n-1}~n^{m-1}=\sum_{k=1}^{m}m^{n-2}~n^{m-k-1}~(n-1)^{k-1}~(m+n-k).
\end{equation}

Let $\mathfrak{T}_{m+n}\subseteq \mathcal{U}_{K_{m,n}}$ denote the set of all uprooted spanning trees of $K_{m,n}$ with root $m+n$.
Since $m+n$ is greater than all other vertices $i\in V(K_{m,n})$, every spanning tree in ${\rm SPT}(K_{m,n})$ becomes an uprooted spanning tree of $K_{m,n}$ when rooted at $m+n$.
Therefore, we have 
\begin{equation}\label{eq:SPT(K_mn)}
   |\mathfrak{T}_{m+n}|=  |{\rm SPT}(K_{m,n})| = m^{n-1}~n^{m-1}.
\end{equation}
Here, we enumerate the elements of $\mathfrak{T}_{m+n}$ by considering the highest child of the root $m+n$.

\begin{theorem}\label{thm:Kmn}
Let $\mathfrak{T}_{m+n}^{k}\subseteq \mathfrak{T}_{m+n}$ be the set of all uprooted spanning trees of $K_{m,n}$ with root $m+n$, such that the highest child of the root equals $m+1-k$, for $1\leq k\leq m$. Then, 
\begin{align*}
|\mathfrak{T}_{m+n}^{k}|=m^{n-2}~n^{m-k-1}~(n-1)^{k-1}~(m+n-k).  
\end{align*}
\end{theorem}

\begin{proof}
Observe that any uprooted spanning tree in $\mathfrak{T}_{m+n}^{k}$ has root $m+n$, and this root is not adjacent to any vertex in the set $\{m+2-k,m+3-k,\dots , m\}\subset V(K_{m,n})$.
Therefore, consider the subgraph $G_{m+n}^{k}$ of $K_{m,n}$ defined as $$G_{m+n}^{k} = K_{m,n}\backslash \{ e_{m+n,i}\colon m+2-k\leq i\leq m\},$$
where $e_{u,v}$ denotes the edge $\{u,v\}\in E(K_{m,n})$.
Note that, in any spanning tree of $G_{m+n}^{k}$, the vertex $m+n$ is not adjacent to any vertex in $\{m+2-k,m+3-k,\dots , m\}$. 

Now consider the graph $(G_{m+n}^{k};e)$, where $e=\{m+1-k,m+n\}\in E(G_{m+n}^{k})$, and let us assign a variable weight $x$ to the edge $e$.
Selecting any spanning tree of $(G_{m+n}^{k};e)$ and designating  $m+n$ as its root yields an uprooted spanning tree in $\mathfrak{T}_{m+n}^{k}$.
Thus, we have
\begin{equation}\label{eq:SPTx}
|\mathfrak{T}_{m+n}^{k}|= |{\rm SPT}(G_{m+n}^{k};e)|,
\end{equation}
where ${\rm SPT}(G_{m+n}^{k};e)\coloneqq\{T\in {\rm SPT}(G_{m+n}^{k})\colon e=\{m+1-k,m+n\}\in E(T)\}$.

We now enumerate the spanning trees of $(G_{m+n}^{k};e)$ using the modified matrix tree theorem (Theorem~\ref{thm:modifiedMTT}).
By deleting the $(m+n)$-th row and column from the Laplacian matrix $L_{x}(G_{m+n}^{k};e)$, we obtain the reduced Laplacian matrix $\widetilde{L}_{x}(G_{m+n}^{k};e)$, which has the following block matrix form.
\[
\widetilde{L}_{x}(G_{m+n}^{k};e) = 
\begin{bNiceArray}{c|c}[margin,columns-width=auto]
  A_x & B_x \\
  \hline
 C_x & D_x
\end{bNiceArray}_{(m+n-1)\times(m+n-1)},
\]
where $A_x={\rm diag}[\alpha_1 ,\alpha_2, \dots , \alpha_m]$ with 
\begin{equation*}
    \alpha_i=
    \begin{cases}
    n  & \text{ if } 1\leq i\leq m+k, \\
    n-1+x   & \text{ if } i= m+1-k,\\
    n-1   & \text{ if } m+2-k \leq i \leq m,
    \end{cases}
\end{equation*}  
$B_x$ is an $m\times(n-1)$ matrix with all entries equal to $-1$, 
$C_x$ is an $(n-1)\times m$ matrix with all entries equal to $-1$ and 
$D_x$ is the $(n-1)\times (n-1)$ diagonal matrix ${\rm diag}[m,m, \dots , m]$. 
As $D_x$ is invertible, using the Schur's formula for determinants of block matrices, we have
\[
\det(\widetilde{L}_{x}(G_{m+n}^{k};e))= \det (D_x)\cdot \det (A_x-B_xD_{x}^{-1}C_x).
\]
Since $D_x$ is diagonal, we have $\det(D_x) = m^{n-1}$, and its inverse is given by $D_{x}^{-1}={\rm diag}\left[\frac{1}{m},\frac{1}{m}, \dots , \frac{1}{m}\right]$.
The matrix product $B_xD_{x}^{-1}C_x$ is a rank one matrix where every entry equals $\lambda\coloneqq\left(\frac{n-1}{m}\right)$. Therefore, $A_x-B_xD_{x}^{-1}C_x =[\beta_{ij}]_{1\leq i,j \leq m}$, where all non-diagonal entries are equal to $-\lambda$ and the diagonal entries are given by
\begin{equation*}
    \beta_{ij}=
    \begin{cases}
    n-\lambda  & \text{ if } 1\leq i\leq m-k, \\
    n-\lambda-1+x   & \text{ if } i=m+1-k,\\
    n-\lambda-1  & \text{ if } m+2-k \leq i\leq m.
    \end{cases}
\end{equation*}  

By Theorem~\ref{thm:modifiedMTT}, in order to compute $|{\rm SPT}(G_{m+n}^{k};e)|$, it suffices to extract the coefficient of $x$ in $\det(\widetilde{L}_{x}(G_{m+n}^{k};e))$. 
Noting that $\det(D_x)$ is independent of $x$ and the matrix $A_x-B_xD_{x}^{-1}C_x$ contains $x$ only in the $(m+1-k)$-th diagonal entry, it follows that the coefficient of $x$ in $\det(\widetilde{L}_{x}(G_{m+n}^{k};e))$ is given by
\[
\det (D_x)\cdot\det M_\lambda,
\]
where $M_\lambda$ is the $(m-1)\times (m-1)$ submatrix of $A_x-B_xD_{x}^{-1}C_x$ obtained by deleting the $(m+1-k)$-th row and column. In fact, 
\begin{equation*}
M_\lambda =
\begin{blockarray}{*{7}{c}r}
 & &  &  \substack{C_{m-k} \\ \downarrow} &  &  &  & \\
\begin{block}{[*{7}{c}]r}
 n-\lambda & -\lambda & \cdots & -\lambda & -\lambda &\cdots  & -\lambda  & \\
 -\lambda & n-\lambda & \cdots & -\lambda   & -\lambda  & \cdots & -\lambda  & \\
 \vdots & \vdots  & \ddots & \vdots & \vdots  &  \ddots & \vdots  & \\
 -\lambda  & -\lambda   & \cdots & n-\lambda  & -\lambda   &  \cdots & -\lambda  & \leftarrow \substack{R_{m-k} \\ }  \\
 -\lambda  & -\lambda   & \cdots & -\lambda  & n-\lambda-1   & \cdots& -\lambda & \\
 \vdots & \vdots    &\ddots  &\vdots & \vdots & \ddots &\vdots & \\
 -\lambda & -\lambda  & \cdots & -\lambda  & -\lambda  & \cdots & n-\lambda-1 & {\scriptstyle (m-1) \times (m-1)}& \\
\end{block}
\end{blockarray}.
\end{equation*}
Let $R_i$ and $C_j$ denote the $i$-th row and $j$-th column of $M_\lambda$, respectively.
We now perform the following sequence of row and column operations on $M_\lambda$.
\begin{enumerate}
    \item $C_{1} \longrightarrow C_{1} + \sum_{i=2}^{m-1} C_{i}$,
    \item $R_i \longrightarrow R_i - R_{1}$, for all $ 2\leq i\leq  m-1$,
    \item $C_{1} \longrightarrow C_{1} + \frac{1}{n-1} C_i$, for all $m-k+1\leq i\leq  m-1$.
\end{enumerate}
After applying these operations in the stated order, the matrix $M_\lambda$ is transformed into an upper triangular matrix $M_\lambda^{\prime}=[m^{\prime}_{ij}]_{1\leq i,j\leq (m-1)}$, whose diagonal entries are given by
\begin{equation*}
    m^{\prime}_{ii}=
    \begin{cases}
    \lambda+1- \left(\frac{k-1}{m}\right) & \text{ if } i= 1, \\
    n & \text{ if } 2\leq i\leq m-k,\\
    n-1  & \text{ if } m-k+1 \leq i\leq m-1.
    \end{cases}
\end{equation*} 
Since the determinant of an upper triangular matrix is the product of its diagonal entries, and all the above operations preserve the determinant, we have
\begin{align*}
\det(M_\lambda) = \det(M_\lambda^{\prime}) &= \left(\lambda+1- \left(\frac{k-1}{m}\right)\right)~ n^{m-k-1}~(n-1)^{k-1}\\
& =  m^{-1}~(m+n-k)~(n-1)^{k-1}~n^{m-k-1}.
\end{align*}
Thus, by Theorem~\ref{thm:modifiedMTT}, we obtain 
\[
|{\rm SPT}(G_{m+n}^{k};e)|  = \det (D_x)\cdot\det M_\lambda=  m^{n-2}~(m+n-k)~(n-1)^{k-1}~n^{m-k-1}.
\]
From equation \eqref{eq:SPTx}, it follows that
\[
|\mathfrak{T}_{m+n}^k|= |{\rm SPT}(G_{m+n}^{k};e)|= m^{n-2}~(m+n-k)~(n-1)^{k-1}~n^{m-k-1}.
\]
This completes the proof.
\end{proof}

We now note that $\mathfrak{T}_{m+n} =  \coprod_{k=1}^{m}  \mathfrak{T}_{m+n}^{k}$, and hence by \eqref{eq:SPT(K_mn)},
\begin{equation}\label{eq:forid2}
m^{n-1}~n^{m-1} = \sum_{k=1}^{m}m^{n-2}~(m+n-k)~(n-1)^{k-1}~n^{m-k-1}.    
\end{equation}


\section{\texorpdfstring{Refined enumeration and associated combinatorial identities}{Refined enumeration of uprooted spanning trees in Kn}}\label{sec:refinement}

In Section~\ref{sec:Kn}, we derived a formula for the total number of uprooted spanning trees of the complete graph $K_n$ rooted at a given vertex. In this section, we refine that enumeration by classifying the uprooted spanning trees based on the highest child of the root. 
Recall that $\mathcal{T}_{n-k}\subseteq \mathcal{U}_{K_{n}}$ denotes the set of all uprooted spanning trees of $K_n$ with root $n-k$.
For $1 \leq j \leq n-k-1$, let $\mathcal{T}_{n-k}^{j}\subseteq\mathcal{T}_{n-k}$ denote the subset consisting of those trees in which the highest child of the root $n-k$ is exactly $n-k-j$.
Then we have the disjoint union $\mathcal{T}_{n-k}=  \coprod_{j=1}^{n-k-1} \mathcal{T}_{n-k}^{j}$, and hence,
\begin{equation}\label{eq:Tj}
|\mathcal{T}_{n-k}|=  \sum\limits_{j=1}^{n-k-1} |\mathcal{T}_{n-k}^{j}|.
\end{equation}
The number of uprooted spanning trees in each subset $\mathcal{T}_{n-k}^{j}$ is given by the following result.
\begin{theorem}\label{Thm:refinement}
Let $\mathcal{T}_{n-k}^{j}\subseteq\mathcal{T}_{n-k}$ be the set of all uprooted spanning trees of $K_n$ with root $n-k$, such that the highest child of the root equals $n-k-j$ for $1\leq j\leq n-k-1$. Then, 
\[
|\mathcal{T}_{n-k}^{j}|= n^{n-k-j-2}~(n-1)^{k+j-2}~(2n-k-j-1).
\]
\end{theorem}

\begin{proof}
For any uprooted spanning tree in $\mathcal{T}_{n-k}^{j}$, the root $n-k$ is not adjacent to any vertex $u$ where $u>n-k$, nor to any vertex $v$ satisfying $n-k-j+1 \leq v \leq n-k-1$.
To encode this structural restriction, consider the graph $$G_{n-k}^{j}=K_n\setminus\{e_{n-k,v}\colon n-k-j+1\leq v \leq n, v\neq n-k\},$$ where $e_{u,v}$ denotes the edge $\{u,v\}\in E(K_{n})$.
Now, let $e\in E(G_{n-k}^{j})$ be the edge between the root $n-k$ and its highest child $n-k-j$, which must appear in every tree in $\mathcal{T}_{n-k}^{j}$. 
We assign a variable weight $x$ to this edge and denote the resulting graph, with the distinguished edge $e = \{n-k-j, n-k\} \in E(G_{n-k}^{j})$, as $(G_{n-k}^{j};e)$.
Note that for each $1 \leq j \leq n-k-1$, the degrees of the vertices in $G_{n-k}^{j}$ are given by
\begin{equation*}
    \deg_{G_{n-k}^{j}}(u)=
    \begin{cases}
    n-1  & \text{ if } u\in[n-k-j-1], \\
    n-2+x  & \text{ if } u= n-k-j,\\
    n-k-j-1+x  & \text{ if } u= n-k,\\
    n-2  & \text{ if } u\in \{n-k-j+1, \dots, n\} \setminus \{n-k\}.
    \end{cases}
\end{equation*} 

Choosing any spanning tree of $G_{n-k}^{j}$ that contains the edge $e$, and designating $n-k$ as the root, yields an uprooted spanning tree in $\mathcal{T}_{n-k}^{j}$.
Thus, we have
\begin{equation}\label{eq:SPTj}
|\mathcal{T}_{n-k}^{j}|= |{\rm SPT}(G_{n-k}^{j};e)|.    
\end{equation}
To compute the number of spanning trees of the graph $(G_{n-k}^{j};e)$, we now apply the modified matrix tree theorem (Theorem~\ref{thm:modifiedMTT}).
Let $L_x(G_{n-k}^{j};e)$ denotes the Laplacian matrix of the graph $(G_{n-k}^{j};e)$.
By deleting the $(n-k)$-th row and column from $L_x(G_{n-k}^{j};e)$, we obtain the reduced Laplacian matrix $\widetilde{L}_x(G_{n-k}^{j};e)$.
This matrix has all non-diagonal entries equal to $-1$, and the diagonal entries are given by the degrees of the corresponding vertices in the graph $(G_{n-k}^{j};e)$.
Let $R_i$ and $C_{i^{\prime}}$ denote the $i$-th row and $i^{\prime}$-th column of $(G_{n-k}^{j};e)$, respectively.
We now perform the following row and column operations on $L_x(G_{n-k}^{j};e)$ to simplify its structure.
\begin{enumerate}
    \item $R_{n-1} \longrightarrow R_{n-1} + \sum_{i=1}^{n-2} R_{i}$,
    \item $C_i \longrightarrow C_i - C_{n-1}$, for all $1\leq i\leq n-2$,
    \item $R_{n-1} \longrightarrow R_{n-1} - \frac{1}{n} R_i$, for each $1\leq i\leq n-k-j-1$,
    \item $C_{n-1} \longrightarrow C_{n-1} + \frac{1}{n} C_i$, for each $1\leq i\leq n-k-j-1$, and then
    \item $C_{n-1} \longrightarrow C_{n-1} + \frac{1}{n-1} C_i$, for each $n-k-j+1\leq i\leq  m-2$.
\end{enumerate}
After applying these operations in the stated order, the matrix $\widetilde{L}_x(G_{n-k}^{j};e)$ reduces to the following block matrix.
\[
\widetilde{L}_x(G_{n-k}^{j};e)^{\prime} = 
\begin{bNiceArray}{c|c}[margin,columns-width=auto]
  A^{\prime}_x & B^{\prime}_x \\
  \hline
 C^{\prime}_x & D^{\prime}_x
\end{bNiceArray}_{(n-1)\times(n-1)},
\]
where $A^{\prime}_x$ is the  $(n-k-j)\times(n-k-j)$ diagonal matrix ${\rm diag}[n ,\dots n,  , n-1+x]$,  
$B^{\prime}_x$ is an $(n-k-j)\times(k+j-1)$ matrix with all entries equal to zero except the $(n-k-j, k+j-1)$-th entry, which is equal to $-1$, 
$C^{\prime}_x$ is a $(k+j-1)\times (n-k-j)$ matrix with a single non-zero entry $x$ at its $(k+j-1,n-k-j)$-th position and 
$D^{\prime}_x$ is the $(k+j-1)\times (k+j-1)$ diagonal matrix ${\rm diag}\left[n-1, \dots ,n-1 ,\frac{n-k-j-1}{n}\right]$. 

Since all of the above row and column operations preserve the determinant, and using the Schur's formula for determinants of block matrices, we obtain $$\det(\widetilde{L}_{x}(G_{m+n}^{k};e))= \det (D^{\prime}_x)\cdot \det (A^{\prime}_x-B^{\prime}_x(D^{\prime}_{x})^{-1}C^{\prime}_x).$$
Since $D^{\prime}_x$ is diagonal, we compute $\det(D^{\prime}_x) = (n-1)^{k+j-2}\left(\frac{n-k-j-1}{n}\right)$, and its inverse is given by $(D^{\prime}_{x})^{-1}={\rm diag}[\frac{1}{n-1}, \dots ,\frac{1}{n-1} , \frac{n}{n-k-j-1}]$.
Further, 
$A^{\prime}_x-B^{\prime}_x(D^{\prime}_{x})^{-1}C^{\prime}_x =  {\rm diag}\left[n, \dots ,n ,(n-1)+\gamma x\right]$, where $\gamma = \frac{2n-k-j-1}{n-k-j-1}$.
Hence,
\[
\det(A^{\prime}_x-B^{\prime}_x(D^{\prime}_{x})^{-1}C^{\prime}_x) = n^{n-k-j-1}\left(n-1+\gamma x\right).
\]
Combining the above expressions, we obtain
\[
\det(\widetilde{L}_{x}(G_{m+n}^{k};e))=  \det (D^{\prime}_x) \left(n^{n-k-j-1}\left(n-1+\gamma x\right)\right).
\]
Therefore, by Theorem~\ref{thm:modifiedMTT}, the number of spanning trees of $(G_{n-k}^{j};e)$ that contain the edge $e$ is
$|{\rm SPT}(G_{n-k}^{j};e)|= \det (D^{\prime}_x) (\gamma ~ n^{n-k-j-1} ).$
Substituting the values of $\det (D^{\prime}_x)$ and $\gamma$, followed by \eqref{eq:SPTj}, we get
\[
|\mathcal{T}_{n-k}^{j}|= |{\rm SPT}(G_{n-k}^{j};e)| = n^{n-k-j-2}~ (n-1)^{k+j-2}~(2n-k-j-1).
\]
This concludes the proof.
\end{proof}

By the decomposition $\mathcal{T}_{n-k}=  \coprod_{j=1}^{n-k-1} \mathcal{T}_{n-k}^{j}$, we obtain the following identity.

\begin{proposition}\label{prop:refinement}
For $n\geq 2$ and $0\leq k\leq n-1$, let $\mathcal{T}_{n-k}$ be the set of all uprooted spanning trees of $K_n$ with root $n-k$. Then, as $|\mathcal{T}_{n-k}| = (n-1-k)~n^{n-2-k}~(n-1)^{k-1}$, we have
\begin{equation}\label{eq:prop5.2}
(n-1-k)~n^{n-2-k}~(n-1)^{k-1}= \sum_{j=1}^{n-k-1} n^{n-k-j-2}~ (n-1)^{k+j-2}~(2n-k-j-1).   
\end{equation}
\end{proposition}

\begin{proof}
The proof follows directly from \eqref{eq:Tj} and Theorem~\ref{Thm:refinement}.
\end{proof}

Using Proposition~\ref{prop:refinement}, we obtain the following refinement of identity~\eqref{eq:identity1}
\begin{equation}\label{eq:sumrefine}
(n-1)^{n-1}= \sum_{k=0}^{n-1}~\sum_{j=1}^{n-k-1} n^{n-k-j-2}~ (n-1)^{k+j-2}~(2n-k-j-1).    
\end{equation}

We now summarize the combinatorial identities obtained in this paper and present their equivalent forms.
Dividing both sides of the identity~\eqref{eq:identity1} by $(n-1)^{n-2}$, we obtain the following equivalent identity
\begin{equation}\label{eq:rmrk1}
n-1= \sum_{k=0}^{n-1}\left(\frac{n}{n-1}\right)^{n-2-k}\left(1-\frac{k}{n-1}\right),\text{ for } n\geq 2.    
\end{equation}
Similarly, dividing both sides of identity~\eqref{eq:identity2} by $m^{n-2}~n^{m-1}$, we obtain an equivalent form of \eqref{eq:identity2} as
\begin{equation}\label{eq:rmrk2}
m= \sum_{j=1}^{m}\left(\frac{n-1}{n}\right)^{j-1}\left(\frac{m+n-j}{n}\right),\text{ for } n\geq 2,\ m\geq 1.        
\end{equation}
Next, we consider the identity~\eqref{eq:prop5.2}.
On dividing both sides by $n^{n-k-2}~(n-1)^{k-1}$, we obtain
\begin{equation}\label{eq:rmrk3}
n-1-k = \sum_{j=1}^{n-k-1}\left(\frac{n-1}{n}\right)^{j-1}\left(\frac{2n-k-j-1}{n}\right),\text{ for } n\geq 2 \text{ and } 0\leq k\leq n-1.        
\end{equation}
By setting $m = n - 1 - k$, we see that identity~\eqref{eq:rmrk3} is equivalent to identity~\eqref{eq:rmrk2}.
Both identities \eqref{eq:rmrk1} and \eqref{eq:rmrk2} can be easily verified by induction.

From these observations, we state the following proposition which collects the key identities in simplified form.

\begin{proposition}\label{prop:identities}
We have the following combinatorial identities:
\begin{enumerate}[label=(\arabic*)]
    \item \label{1} For $n \geq 2$,  
    \[
    n-1 = \sum_{k=0}^{n-1} \left(\frac{n}{n-1}\right)^{n-2-k} \left(1 - \frac{k}{n-1}\right).
    \]
    
    \item \label{2} For $n \geq 2$, $m \geq 1$,  
    \[
    m = \sum_{j=1}^{m} \left(\frac{n-1}{n}\right)^{j-1} \left(\frac{m+n-j}{n}\right).
    \]
\end{enumerate}   
\end{proposition}

\begin{proof}
We first prove the initial identity by induction on $n$.  
Here, we introduce an auxiliary real parameter $a \in \mathbb{R}\setminus\{0,1\}$ and consider the more general identity
\begin{equation}\label{eq:induction_a}
 \sum_{k=0}^{n-1} \left(\frac{a}{a-1}\right)^{n-2-k}\left(1-\frac{k}{a-1}\right)
 = (n-1) + \left(\frac{a-n}{a}\right), \qquad \forall\  a\in \mathbb{R}\setminus\{0,1\}.
\end{equation}
We verify \eqref{eq:induction_a} by induction on $n$, and upon substituting $a=n$, we obtain the desired identity for all integers $n\geq 2$.

The second identity can be proved directly by induction on $m$, without introducing any auxiliary parameter.
\end{proof}

\begin{remark}
Using the first identity in Proposition~\ref{prop:identities} together with Theorem~\ref{thm:Kn}, one recovers the formula $|\mathcal{U}_{K_{n}}| = (n-1)^{n-1}$ for the number of uprooted spanning trees of the complete graph $K_n$, as an application of the matrix tree theorem.

In the following section, the matrix tree theorem will again be used to compute the number of uprooted spanning trees of another graph obtained from $K_n$.
\end{remark}


\section{\texorpdfstring{Another application of the matrix tree theorem}{Another application of matrix tree theorem}}\label{sec:more_applications}

In this section, we apply the matrix tree theorem to enumerate uprooted spanning trees of certain graphs derived from the complete graph $K_{n}$.
In particular, we consider the graph $K_{n}\setminus \{e_{1,n}\}$ obtained by deleting the edge $\{1,n\}\in E(K_{n})$, and compute the corresponding number of uprooted spanning trees.

For $n\geq 3$, let $\mathcal{U}_{K_n}^{\prime}$ denote the set of all uprooted spanning trees of the graph $K_{n}\setminus \{e_{1,n}\}$, where $e_{1,n}$ denotes the edge $\{1,n\}\in E(K_n)$. 
Kumar et al.~\cite{CGS} established that
\begin{equation}\label{eq:kn-e}
|\mathcal{U}_{K_n}^{\prime}|  = (n-1)^{n-3}~(n-2)^{2}.
\end{equation}
The proof in~\cite{CGS} is quite involved and requires many detailed counting steps.
In this section, we obtain the same result more directly as an application of the matrix tree theorem.

Let $\mathcal{T}^{\prime}_{r}\subseteq\mathcal{U}_{K_n}^{\prime}$ denote the set of all uprooted spanning trees of $K_n\setminus\{e_{1,n}\}$ with root $r$.
For $n\geq 3$ and $0\leq k\leq n-2$, we have $\mathcal{U}_{K_n}^{\prime}= \coprod_{k=0}^{n-2} \mathcal{T}^{\prime}_{n-k}$.
Hence,
\begin{equation}\label{eq:sec5_sum}
|\mathcal{U}_{K_n}^{\prime}|= \sum_{k=0}^{n-2} \mathcal{T}^{\prime}_{n-k}.
\end{equation}
In what follows, we focus on computing $|\mathcal{T}_{r}^{\prime}|$ by applying the matrix tree theorem.
 
\begin{theorem}\label{thm:Kn/e}
For $n\geq 3$, let $\mathcal{T}^{\prime}_{r}\subseteq \mathcal{U}^{\prime}_{K_n}$ be the set of all the uprooted spanning trees of $K_n\setminus\{e_{1,n}\}$ with root $r$. Then we have
\[
|\mathcal{T}^{\prime}_{n}|=n^{n-3}~(n-2), 
\]
and for $1\leq k\leq n-2$,
\[
|\mathcal{T}^{\prime}_{n-k}|=n^{n-2-k}~(n-1)^{k-1}\left((n-3-k)+\frac{2nk+n-k-2}{n(n-1)}\right).
\]
\end{theorem}

\begin{proof}
Let us consider $\mathcal{T}^{\prime}_{n}$, the set of all uprooted spanning trees of $K_n\setminus\{e_{1,n}\}$ with root $n$. 
Choosing any spanning tree of $K_n\setminus\{e_{1,n}\}$ and designating $n$ as the root yields in an uprooted spanning tree of $K_n\setminus\{e_{1,n}\}$ with root $n$. 
Therefore, we obtain 
\begin{equation}\label{eq:sec5a}
      |\mathcal{T}^{\prime}_{n}|=|{\rm SPT}(K_n\setminus\{e_{1,n}\})|.
\end{equation}

To determine $|{\rm SPT}(K_n\setminus\{e_{1,n}\})|$, we apply the matrix tree theorem.
We obtain the reduced Laplacian matrix of $K_n\setminus\{e_{1,n}\}$ by deleting the $n$-th row and column from the Laplacian matrix $L(K_n\setminus\{e_{1,n}\})$.
The resulting reduced Laplacian matrix $\widetilde{L}(K_n\setminus\{e_{1,n}\})$ is an $(n-1)\times (n-1)$ matrix in which the first diagonal entry is $n-2$, all remaining diagonal entries are $n-1$, and all non-diagonal entries are $-1$.

Let $R_{i}$ denote the $i$-th row and $C_{j}$ denote the $j$-th column of the matrix $\widetilde{L}(K_n\setminus\{e_{1,n}\})$.
Let us perform the following series of row and column operations on $\widetilde{L}(K_n\setminus\{e_{1,n}\})$.
\begin{enumerate}
    \item $C_1 \longrightarrow C_1 + \sum_{i=2}^{n-1} C_{i}$,
    \item $R_i \longrightarrow R_i - R_{1}$, for all $ 2\leq i\leq  n-1$,
    \item $C_1 \longrightarrow C_1 -\frac{1}{n} C_i$, for each $2\leq i\leq  n-1$.
\end{enumerate}
After applying these operations in the given order, the matrix $\widetilde{L}(K_n\setminus\{e_{1,n}\})$ is reduced to an upper triangular matrix with the first diagonal entry $\frac{n-2}{n}$, all remaining diagonal entries equal to $n$.
Since the determinant of an upper triangular matrix is the product of its diagonal entries, and all the above operations preserve the determinant, we have
\[{\rm det}(\widetilde{L}(K_n\setminus\{e_{1,n}\}))=n^{n-2}~\bigg(\frac{n-2}{n}\bigg) = n^{n-3}~(n-2).\]
By the matrix tree theorem, the number of spanning trees of $K_n\setminus\{e_{1,n}\}$ is  
\begin{equation}\label{eq:sec5b}
     |{\rm SPT}(K_n\setminus\{e_{1,n}\})| = {\rm det}(\widetilde{L}(K_n\setminus\{e_{1,n}\}))=n^{n-3}~(n-2).
\end{equation}
Thus, from \eqref{eq:sec5a} and \eqref{eq:sec5b}, we conclude 
\begin{align*}
|\mathcal{T}^{\prime}_{n}|=|{\rm SPT}(K_n\setminus\{e_{1,n}\})| = n^{n-3}~(n-2).
\end{align*}

Now, consider $\mathcal{T}^{\prime}_{n-k}$, the set of all uprooted spanning trees of $K_n\setminus\{e_{1,n}\}$ with root $n-k$, where $1\leq k\leq n-1$.
For any tree $T\in\mathcal{T}^{\prime}_{n-k}$, observe that, in addition to the condition that vertex $1$ is not adjacent to vertex $n$, the root $n-k$ is not adjacent to any vertex $i$ where $i > n-k$. 
To reflect this, we define the graph $$G^{\prime}_{n-k}= K_n\backslash \{e_{1,n},e_{n-k,i}\colon n-k+1\leq i\leq n\},$$ where $e_{u,v}$ represents the edge $\{u,v\}\in E(K_n)$.
By construction, $G^{\prime}_{n-k}$ ensures that vertex $1$ is not adjacent to vertex $n$, and vertex $n-k$  is not adjacent to any vertex  $i$ for $ n-k+1\leq i\leq n$.
Thus, choosing any spanning tree of $G^{\prime}_{n-k}$ and designating $n-k$ as the root gives an uprooted spanning tree of $K_n\setminus\{e_{1,n}\}$ with root $n-k$. 
Therefore, we obtain 
\begin{equation}\label{eq:sec5c}
      |\mathcal{T}^{\prime}_{n-k}|=|{\rm SPT}(G^{\prime}_{n-k})|, \ \text{ for all } 1\leq k\leq n-2.
\end{equation}

We now determine $|{\rm SPT}(G^{\prime}_{n-k})|$ using the matrix tree theorem.
The reduced Laplacian matrix of $G^{\prime}_{n-k}$ is obtained by deleting the $(n-k)$-th row and column from the Laplacian matrix $L(G^{\prime}_{n-k})$.
The resulting reduced Laplacian matrix $\widetilde{L}(G^{\prime}_{n-k})$ is an $(n-1)\times (n-1)$ matrix in which the first diagonal entry is $n-2$, 
the last diagonal entry is $n-3$,
the $i$-th diagonal entry equals $n-1$ for $2 \leq i \leq n-k-1$ and $n-2$ for $n-k \leq i \leq n-2$, while all non-diagonal entries are $-1$ except for the $(1,n-1)$-th and $(n-1,1)$-th entries, which are $0$.

Let $R_{i}$ and $C_{j}$ denote the $i$-th row and $j$-th column of the matrix $\widetilde{L}(G^{\prime}_{n-k})$, respectively.
We now perform the following sequence of row and column operations on $\widetilde{L}(G^{\prime}_{n-k})$.
\begin{enumerate}
    \item $C_1 \longrightarrow C_1 + \sum_{i=2}^{n-1} C_{i}$,
    \item $R_i \longrightarrow R_i - R_{1}$, for all $ 2\leq i\leq  n-k-1$,
    \item $C_i \longrightarrow C_i + C_1$, for all $2\leq i\leq  n-1$,
    \item $R_i \longrightarrow R_i - R_{n-1}$, for all $ n-k\leq i\leq  n-2$,
    \item $R_{n-1} \longrightarrow R_{n-1} +\frac{1}{n} R_{i}$, for all $ 2\leq i\leq  n-k-1$, and then
    \item $R_{n-1} \longrightarrow R_{n-1} + \frac{1}{n-1}R_{j}$, for all $n-k\leq j\leq  n-2$.
\end{enumerate}
After applying the above operations in the given order, the matrix $\widetilde{L}(G^{\prime}_{n-k})$ reduces to an upper triangular matrix whose first diagonal entry is $1$, 
the last diagonal entry is $(n-3-k) +\frac{2nk+n-k-2}{n(n-1)}$,
the $i$-th diagonal entry equals $n$ for $2 \leq i \leq n-k-1$ and equals $n-1$ for $n-k \leq i \leq n-2$.
Since the determinant of an upper triangular matrix is the product of its diagonal entries, and all the above operations preserve the determinant, we obtain
\[{\rm det}(\widetilde{L}(G^{\prime}_{n-k}))=n^{n-k-2}~(n-1)^{k-1}~\bigg((n-3-k)+\frac{2nk+n-k-2}{n(n-1)}\bigg).\]
By the matrix tree theorem, we have
\begin{equation}\label{eq:sec5d}
     |{\rm SPT}(G^{\prime}_{n-k})| = {\rm det}(\widetilde{L}(G^{\prime}_{n-k}))=n^{n-k-2}~(n-1)^{k-1}~\bigg((n-3-k)+\frac{2nk+n-k-2}{n(n-1)}\bigg).
\end{equation}
Hence, from \eqref{eq:sec5c} and \eqref{eq:sec5d}, we have, for all $1\leq k\leq n-2$,
\begin{align*}
|\mathcal{T}^{\prime}_{n-k}|=|{\rm SPT}(G^{\prime}_{n-k})| = n^{n-k-2}~(n-1)^{k-1}~\bigg((n-3-k)+\frac{2nk+n-k-2}{n(n-1)}\bigg). 
\end{align*}
This concludes our proof.
\end{proof}

By the decomposition $\mathcal{U}_{K_n}^{\prime}= \coprod_{k=0}^{n-2} \mathcal{T}^{\prime}_{n-k}$, we obtain 
\begin{equation}\label{eq:prop_sec05}
|\mathcal{U}_{K_n}^{\prime}|=n^{n-3}~(n-2)+\sum_{k=1}^{n-2} n^{n-2-k}~(n-1)^{k-1}\left((n-3-k)+\frac{2nk+n-k-2}{n(n-1)}\right),   
\end{equation}
which follows directly from \eqref{eq:sec5_sum} and Theorem~\ref{thm:Kn/e}.

%
%

\begin{proposition}\label{prop:identity_Kn-e}
For $n\geq 3$, the following combinatorial identity holds:
\begin{equation}\label{eq:rmrkSec5}
n-2 = \left(\frac{n}{n-1}\right)^{n-3} + \sum_{k=1}^{n-2} \left(\frac{n}{n-1}\right)^{n-2-k} \left( \frac{n-3-k}{n-2} + \frac{2nk+n-k-2}{n(n-1)(n-2)} \right).
\end{equation}
\end{proposition}

\begin{proof}
Identity~\eqref{eq:rmrkSec5} can be verified by induction on $n$. 
In fact we prove a more general identity.
For a real parameter $a\in \mathbb{R}\setminus\{0,1,2\}$, consider the identity
\begin{equation}\label{eq:induction_a1}
\begin{split}
\left(\frac{a}{a-1}\right)^{n-3} + \sum_{k=1}^{n-2}\left(\frac{a}{a-1}\right)^{n-2-k}\left(\frac{a-3-k}{a-2} + \frac{2ak+a-k-2}{a(a-1)(a-2)}\right) \\ 
= (n-2) + \left(1 + \frac{n-(n-1)a}{a(a-2)}\right).
\end{split}
\end{equation}
Following the approach in Section~\ref{sec:refinement}, this general identity can be proved by induction on $n$.  
Substituting $a=n$ in \eqref{eq:induction_a1} recovers the desired identity~\eqref{eq:rmrkSec5} for all $n\geq 3$.  
\end{proof}

Multiplying both sides of identity~\eqref{eq:prop_sec05} by $(n-1)^{n-3}(n-2)$ we obtain an equivalent form.  
For $n\geq 3$, this gives
\begin{equation}\label{eq:identity_final5}
(n-1)^{n-3}(n-2)^{2}
= n^{n-3}(n-2)
+ \sum_{k=1}^{n-2} n^{n-2-k}(n-1)^{k-1}
\left((n-3-k) + \frac{2nk+n-k-2}{n(n-1)}\right).
\end{equation}
Combining \eqref{eq:identity_final5} with \eqref{eq:prop_sec05}, we recover the formula $|\mathcal{U}_{K_n}^{\prime}| = (n-1)^{n-3}(n-2)^{2}$,
which gives the number of uprooted spanning trees of the graph $K_n \setminus \{e_{1,n}\}$ obtained by deleting the edge $\{1,n\}\in E(K_n)$.  
This provides an alternative and more direct proof of the result of Kumar et al.~\cite{CGS}, as an application of the matrix tree theorem, thereby avoiding the elaborate counting arguments in their original approach.

\begin{remark}
In the course of proving Proposition~\ref{prop:identities} and Proposition~\ref{prop:identity_Kn-e}, two general identities, \eqref{eq:induction_a} and \eqref{eq:induction_a1}, were encountered.
For an integer $a> n$, it would be interesting to see whether these identities possess a combinatorial interpretation.

Furthermore, the observations made in this article lead to a natural open question:  
\emph{Does there exist an analogue of the matrix tree theorem for computing the number of uprooted spanning trees of an arbitrary graph?}
\end{remark}


\section*{Acknowledgements}
The first author expresses gratitude for the financial assistance received under the Institute Postdoctoral Fellowship offered by the Indian Institute of Science Education and Research (IISER) Mohali.


%
%
%
%
%

\bibliography{main}

\end{document}